\documentclass[11pt]{article}

\usepackage{amsmath}
\usepackage{amssymb}
\usepackage{amsfonts}
\usepackage{amsthm}
\usepackage[margin=1in]{geometry}
\usepackage{hyperref}
\usepackage[width=12cm,format=plain,labelfont=sf,bf,small,textfont=sf,up,small,justification=justified,labelsep=period]{caption}
\usepackage{graphicx}
\usepackage{framed}
\usepackage{xcolor}
\usepackage{tabularx}

\hypersetup{
    colorlinks=true,       
    linkcolor=blue,          
    citecolor=blue,        
    filecolor=blue,      
    urlcolor=blue           
}

\makeatletter
\def\iddots{\mathinner{\mkern1mu\raise\p@
\vbox{\kern7\p@\hbox{.}}\mkern2mu
\raise4\p@\hbox{.}\mkern2mu\raise7\p@\hbox{.}\mkern1mu}}
\makeatother

\newcommand{\RR}{\mathbb{R}}

\DeclareMathOperator{\rank}{rank}

\DeclareMathOperator{\ranksoc}{\rank_{soc}}

\DeclareMathOperator{\ext}{ext} 

\DeclareMathOperator{\covsoc}{cov_{soc}}

\newcommand{\eqsupp}{\overset{supp}{=}} 

\newcommand{\SOC}{\mathcal Q}
\newcommand{\SOCn}{soc}

\newcommand{\E}{E} 
\renewcommand{\S}{\mathbf{S}}

{\begin{small}\begin{list}{}%
         {\setlength{\leftmargin}{1cm}\setlength{\rightmargin}{1cm}}%
         \item[]%
}
{\end{list}\end{small}}

\theoremstyle{plain}

\newtheorem{lemma}{Lemma}
\newtheorem{theorem}{Theorem}

\newtheorem{fact}{Fact}

\theoremstyle{definition}
\newtheorem{definition}{Definition}

\theoremstyle{remark}

\newtheorem{remark}{Remark}

\title{On representing the positive semidefinite cone using the second-order cone}
\author{Hamza Fawzi\thanks{Department of Applied Mathematics and Theoretical Physics, University of Cambridge. Email: \texttt{h.fawzi@damtp.cam.ac.uk}.}}

\date{\today}

\begin{document}

\maketitle

\begin{abstract}
The second-order cone plays an important role in convex optimization and has strong expressive abilities despite its apparent simplicity. Second-order cone formulations can also be solved more efficiently than semidefinite programming in general. We consider the following question, posed by Lewis and Glineur, Parrilo, Saunderson: is it possible to express the general positive semidefinite cone using second-order cones? We provide a negative answer to this question and show that the $3\times 3$ positive semidefinite cone does not admit any second-order cone representation. Our proof relies on exhibiting a sequence of submatrices of the slack matrix of the $3\times 3$ positive semidefinite cone whose ``\emph{second-order cone rank}'' grows to infinity.
We also discuss the possibility of representing certain slices of the $3\times 3$ positive semidefinite cone using the second-order cone.
\end{abstract}

\section{Introduction}

Let $\SOC \subset \RR^3$ denote the three-dimensional \emph{second-order} cone (also known as the ``ice-cream'' cone or the Lorentz cone):
\[
\SOC = \{ (x,t) \in \RR^2 \times \RR : \|x\| \leq t \}.
\]
It is known that $\SOC$ is linearly isomorphic to the cone of $2\times 2$ real symmetric positive semidefinite matrices. Indeed we have:
\begin{equation}
\label{eq:SOC2x2PSD}
(x_1,x_2,t) \in \SOC \quad \Longleftrightarrow \quad \begin{bmatrix} t-x_1 & x_2\\ x_2 & t+x_1\end{bmatrix} \succeq 0.
\end{equation}
Despite its apparent simplicity the second-order cone $\SOC$ has strong expressive abilities and allows us to represent various convex constraints that go beyond ``simple quadratic constraints''. For example it can be used to express geometric means ($x\mapsto \prod_{i=1}^n x_i^{p_i}$ where $p_i \geq 0$ and $\sum_{i=1}^n p_i = 1$), $\ell_p$-norm constraints, multifocal ellipses, robust counterparts of linear programming, etc. We refer the reader to \cite{ben2001lectures} for more details.

Given this strong expressive ability one may wonder whether the general positive semidefinite cone can be represented using $\SOC$. This question was posed in particular by Adrian Lewis (personal communication) and Glineur, Parrilo and Saunderson  \cite{glineurSOC_ICCOPT2013}. In this paper we show that this is not possible, even for the $3\times 3$ positive semidefinite cone. To make things precise we use the language of lifts (or extended formulations), see \cite{gouveia2011lifts}.
We say that a convex cone $K \subset \RR^m$ has a \emph{second-order cone lift of size $k$} (or simply $\SOC^k$-lift) if it can be written as the projection of an affine slice of the Cartesian product of $k$ copies of $\SOC$, i.e.:
\begin{equation}
\label{eq:KSOClift}
K = \pi\left(\SOC^k \cap L\right)
\end{equation}
where $\pi:\RR^{3k} \rightarrow \RR^m$ is a linear map and $L$ is a linear subspace of $\RR^{3k}$, and $\SOC^k$ is the Cartesian product of $k$ copies of $\SOC$:
\[
\SOC^k = \SOC \times \dots \times \SOC \quad \text{($k$ copies)}.
\]
In this paper we prove:
\begin{theorem}
\label{thm:main}
The cone $\S^3_+$ does not admit any $\SOC^k$-lift for any finite $k$.
\end{theorem}
Note that higher-dimensional second order cones of the form
\[
\{(x,t) \in \RR^{n} \times t : \|x\| \leq t\}
\]
where $n \geq 3$ can be represented using the three-dimensional cone $\SOC$, see e.g., \cite[Section 2]{ben2001polyhedral}.
Thus Theorem \ref{thm:main} also rules out any representation of $\S^3_+$ using the higher-dimensional second-order cones.
Moreover since $\S^3_+$ appears as a slice of higher-order positive semidefinite cones Theorem \ref{thm:main} also shows that one cannot represent $\S^n_+$, for $n \geq 3$ using second-order cones.

\section{Preliminaries}

The paper \cite{gouveia2011lifts} introduced a general methodology to prove existence or nonexistence of lifts in terms of the \emph{slack matrix} of a cone. In this section we review some of the definitions and results from this paper, and we introduce the notion of \emph{second-order cone factorization} and \emph{second-order cone rank}.

%
Recall that the dual of a cone $K$ living in Euclidean space $\E$ is defined as:
\[
K^* = \{x \in \E : \langle x, y \rangle \geq 0 \quad \forall y \in K \}.
\]
We also denote by $\ext(K)$ the extreme rays of a cone $K$. 
The notion of \emph{slack matrix} plays a fundamental role in the study of lifts.
\begin{definition}[Slack matrix]
The slack matrix of a cone $K$, denoted $S_K$, is a (potentially infinite) matrix where columns are indexed by extreme rays of $K$, and rows are indexed by extreme rays of $K^*$ (the dual of $K$) and where the $(x,y)$ entry is given by:
\begin{equation}
\label{eq:slackK}
S_K[x,y] = \langle x,y \rangle \quad \forall (x,y) \in \ext(K^*) \times \ext(K).
\end{equation}
\end{definition}

Note that, by definition of dual cone, all the entries of $S_K$ are nonnegative.
Also note that an element $x \in \ext(K^*)$ (and similarly $y \in \ext(K)$) is only defined up to a positive multiple. Any choice of scaling gives a valid slack matrix of $K$ and the properties of $S_K$ that we are interested in will be independent of the scaling chosen.

The existence/nonexistence of a second-order cone lift for a convex cone $K$ will depend on whether $S_K$ admits a certain \emph{second-order cone factorization} which we now define.

\begin{definition}[$\SOC^k$-factorization and second-order cone rank]
\label{def:socrank}
Let $S \in \RR^{I\times J}$ be a matrix with nonnegative entries. We say that $S$ has a $\SOC^k$-factorization if there exist vectors $a_i \in \SOC^k$ for $i\in I$ and $b_j \in \SOC^k$ for $j \in J$ such that $S[i,j] = \langle a_i, b_j \rangle$ for all $i \in I$ and $j \in J$. The smallest $k$ for which such a factorization exists will be denoted $\ranksoc(S)$.
\end{definition}

\begin{remark}
\label{rem:socrankdecomp}
It is important to note that the second-order cone rank of any matrix $S$ can be equivalently expressed as the smallest $k$ such that $S$ admits a decomposition
\begin{equation}
\label{eq:Ssocdecomp}
S = M_1 + \dots + M_k
\end{equation}
where $\ranksoc(M_l) = 1$ for each $l=1,\dots,k$ (i.e., each $M_l$ has a factorization $M_l[i,j] = \langle a_i, b_j \rangle$ where $a_i,b_j \in \SOC$). This simply follows from the fact that $\SOC^k$ is the Cartesian product of $k$ copies of $\SOC$.
\end{remark}
We now state the main theorem from \cite{gouveia2011lifts} that we will need. Recall that a cone $K$ is called proper if it is closed, convex, full-dimensional and such that $K^*$ is also full-dimensional.
\begin{theorem}[{Existence of a lift, special case of \cite{gouveia2011lifts}}]
\label{thm:liftslackmatrix}
A proper cone $K$ has a $\SOC^k$-lift if and only if its slack matrix $S_K$ has a $\SOC^k$-factorization.
\end{theorem}

\paragraph{The cone $\S^3_+$} In this paper we are interested in the cone $K = \S^3_+$ of real symmetric $3\times 3$ positive semidefinite matrices. Note that the extreme rays of $\S^3_+$ are rank-one matrices of the form $xx^T$ where $x \in \RR^3$. Also note that $\S^3_+$ is self-dual, i.e., $(\S^3_+)^* = \S^3_+$. The slack matrix of $\S^3_+$ thus has its rows and columns indexed by three-dimensional vectors and
\begin{equation}
\label{eq:slackmatrixS3+}
S_{\S^3_+}[x,y] = \langle xx^T, yy^T \rangle = \left(x^T y\right)^2 \quad \forall (x,y) \in \RR^3 \times \RR^3.
\end{equation}
In order to prove that $\S^3_+$ does not admit a second-order representation, we will show that its slack matrix does not admit any $\SOC^k$-factorization for any finite $k$. In fact we will exhibit a sequence $(A_n)$ of submatrices of $S_{\S^3_+}$ where $\ranksoc(A_n)$ grows to $+\infty$ as $n\rightarrow +\infty$.

Before introducing this sequence of matrices we record the following simple fact concerning orthogonal vectors in the cone $\SOC$ which will be useful later.


\begin{fact}
\label{fact:socorth}
Assume $a,b_1,b_2 \in \SOC$ are all nonzero and $\langle a,b_1 \rangle = \langle a,b_2 \rangle = 0$. Then necessarily $b_1$ and $b_2$ are collinear.
\end{fact}
\begin{proof}
This is easy to see geometrically by visualizing the ``ice cream'' cone. We include a proof for completeness: let $a = (a',t) \in \RR^2 \times \RR$ and $b_i = (b'_i,s_i) \in \RR^2 \times \RR$ where $\|a'\|\leq t$ and $\|b'_i\| \leq s_i$. Note that for $i=1,2$ we have $0 = \langle a, b_i \rangle = \langle a', b'_i \rangle + t s_i \geq - \|a'\| \|b'_i\| + t s_i \geq 0$ where in the first inequality we used Cauchy-Schwarz and in the second inequality we used the definition of the second-order cone. It thus follows that all the inequalities must be equalities: by the equality case in Cauchy-Schwarz we must have that $b'_i = \alpha_i a'$ for some constant $\alpha_i$ and we must also have $t = \|a'\|$ and $s_i = \|b'_i\|$ (note that $\alpha_i < 0$). Thus we get that $b_i = (\alpha_i a', |\alpha_i| \|a'\|) = |\alpha_i| (-a',\|a'\|)$. This shows that $b_1$ and $b_2$ are both collinear to the same vector $(-a',\|a'\|)$ and thus completes the proof.
\end{proof}


\section{Proof}
\label{sec:proof}

\paragraph{A sequence of matrices} We now define our sequence $A_n$ of submatrices of the slack matrix of $\S^3_+$. For any integer $i$ define the vector
\begin{equation}
\label{eq:momcurve}
v_i = (1,i,i^2) \in \RR^3.
\end{equation}
Note that this sequence of vectors satisfies the following:
\begin{equation}
\label{eq:propvis}
\text{For all distinct integers } i_1,i_2,i_3 \;  \det(v_{i_1},v_{i_2},v_{i_3}) \neq 0.
\end{equation}
Our matrix $A_n$ has size $\binom{n}{2}\times n$ and is defined as follows (rows are indexed by 2-subsets of $[n]$ and columns are indexed by $[n]$):
\begin{equation}
A_n[\{i_1,i_2\},j] := \left(( v_{i_1} \times v_{i_2})^T v_j \right)^2 = \det(v_{i_1},v_{i_2},v_j)^2 \quad \forall \{i_1,i_2\} \in \binom{[n]}{2}, \; \forall j \in [n]
\end{equation}
where $\times$ denotes the cross-product of three-dimensional vectors. It is clear from the definition of $A_n$ that it is a submatrix of the slack matrix of $\S^3_+$.
Note that the sparsity pattern of $A_n$ satisfies the following:
\begin{equation}
\label{eq:sparsityAn}
\begin{aligned}
& A_n[e,j] = 0 \text{ if } j \in e\\
& A_n[e,j] > 0 \text{ otherwise}
\end{aligned}
\qquad e \in \binom{[n]}{2}, \; j \in [n].
\end{equation}
Also note that $A_n$ satisfies the following important recursive property: for any subset $C$ of $[n]$ of size $n_0$ the submatrix $A_n[\binom{C}{2},C]$ has the same sparsity pattern as $A_{n_0}$. In our main theorem we will show that the second-order cone rank of $A_n$ grows to infinity with $n$.

\begin{remark}
The specific choice of $v_i$ in \eqref{eq:momcurve} is not important, as long as it satisfies \eqref{eq:propvis}, since the only property that we will use about $A_n$ is its sparsity pattern. For example another choice for $v_i$ that we could use is
\begin{equation}
\label{eq:vistrig}
v_i = (1,\cos(\theta_i),\sin(\theta_i)) \in \RR^3
\end{equation}
where $\theta_1,\theta_2,\ldots$ is any increasing sequence in $[0,2\pi)$. In fact using this choice of $v_i$ we will argue later (Section \ref{sec:slices}) that a certain slice of $\S^3_+$ does not admit a second-order cone lift.
\end{remark}

\paragraph{Covering numbers}
Our analysis of the matrix $A_n$ will only rely on its sparsity pattern. Given two matrices $A$ and $B$ of the same size we write $A \eqsupp B$ if $A$ and $B$ have the same support (i.e., $A_{ij} = 0$ if and only if $B_{ij} = 0$ for all $i,j$). We now define a combinatorial analogue of the \emph{second-order cone rank}:
\begin{definition}
\label{def:covsoc}
Given a nonnegative matrix $A$, we define the $\SOCn$-covering number of $A$, denoted $\covsoc(A)$ to be the smallest number $k$ of matrices $M_1,\dots,M_k$ with $\ranksoc(M_l) = 1$ for $l=1,\dots,k$ that are needed to cover the nonzero entries of $A$, i.e., such that 
\begin{equation}
\label{eq:soccoveringA}
A \eqsupp M_1+\dots+M_k.
\end{equation}
\end{definition}

\begin{fact}
For any nonnegative matrix $A \in \RR^{I\times J}_+$ we have $\ranksoc(A) \geq \covsoc(A)$.
\end{fact}
\begin{proof}
This follows immediately from Remark \ref{rem:socrankdecomp} concerning $\ranksoc$ and the definition of $\covsoc$.
\end{proof}

A simple but crucial fact concerning $\SOCn$-coverings that we will use is the following: in any $\SOCn$-covering of $A$ of the form \eqref{eq:soccoveringA}, each matrix $M_l$ must satisfy $M_l[i,j] = 0$ whenever $A[i,j] = 0$. This is because the matrices $M_1,\dots,M_k$ are all entrywise nonnegative.


We are now ready to state our main result.

\begin{theorem}
\label{thm:covsocAn}
Consider a sequence $(A_n)$ of matrices of sparsity pattern given in \eqref{eq:sparsityAn}.
Then for any $n_0 \geq 2$ we have $\covsoc(A_{3n_0^2}) \geq \covsoc(A_{n_0})+1$. As a consequence $\covsoc(A_{n})\rightarrow +\infty$ when $n\rightarrow +\infty$.
\end{theorem}


The proof of our theorem rests on a key lemma concerning the sparsity pattern of any term in a $\SOCn$-covering of $A_n$. 

\begin{lemma}[Main]
\label{lem:main}
Let $n$ such that $n \geq 3n_0^2$ for some $n_0 \geq 2$. Assume $M \in \RR^{\binom{n}{2} \times n}$ satisfies $\ranksoc(M) = 1$ and $M[e,j] = 0$ for all $e\in \binom{n}{2}$ and $j \in [n]$ such that $j \in e$. Then there is a subset $C$ of $[n]$ of size at least $n_0$ such that the submatrix $M[\binom{C}{2},C]$ is identically zero.
\end{lemma}

Before proving this lemma, we show how this lemma can be used to easily prove Theorem \ref{thm:covsocAn}.

\begin{proof}[Proof of Theorem \ref{thm:covsocAn}]
Let $n = 3n_0^2$ and consider a $\SOCn$-covering of $A_n \eqsupp M_1 + \dots + M_r$ of size $r = \covsoc(A_n)$ (note that we have of course $r \geq 1$ since $A_n$ is not identically zero). By Lemma \ref{lem:main} there is a subset $C$ of $[n]$ of size $n_0$ such that $M_1[\binom{C}{2},C] = 0$. It thus follows that we have $A_n[\binom{C}{2},C] \eqsupp M_2[\binom{C}{2},C] + \dots + M_r[\binom{C}{2},C]$. Also note that $A_n[\binom{C}{2},C] \eqsupp A_{n_0}$. It thus follows that $A_{n_0}$ has a $\SOCn$-covering of size $r-1$ and thus $\covsoc(A_{n_0}) \leq \covsoc(A_{3n_0^2})-1$. This completes the proof.
\end{proof}
For completeness we show how Theorem \ref{thm:main} follows directly from Theorem \ref{thm:covsocAn}.
\begin{proof}[Proof of Theorem \ref{thm:main}]
Since for any $n \geq 1$, $A_n$ is a submatrix of the slack matrix of $\S^3_+$, Theorem \ref{thm:covsocAn} shows that the slack matrix of $\S^3_+$ does not admit any $\SOC^k$-factorization for finite $k$.  This shows, via Theorem \ref{thm:liftslackmatrix}, that $\S^3_+$ does not have a $\SOC^k$-lift for any finite $k$.
\end{proof}

The rest of the paper is devoted to prove Lemma \ref{lem:main}.

\begin{proof}[Proof of Lemma 1]
Let $M \in \RR^{\binom{n}{2}\times n}$ and assume that $M$ has a factorization $M_{e,j} = \langle a_e, b_j \rangle$ where $a_e, b_j \in \SOC$ for all $e \in \binom{[n]}{2}$ and $j \in [n]$, and that $M_{e,j} = 0$ whenever $j \in e$.

Let $E_0 := \{e \in \binom{[n]}{2} : a_e = 0\}$ be the set of rows of $M$ that are identically zero and let $E_1 = \binom{[n]}{2} \setminus E_0$. Similarly for the columns we let $S_0 := \{j \in [n] : b_j = 0\}$ and $S_1 = [n] \setminus S_0$.

In the next lemma we use the sparsity pattern of $A_n$ together with Fact \ref{fact:socorth} to infer additional properties on the sparsity pattern of $M$.

\begin{lemma}
\label{lem:Msparsity}
Let $C$ be a connected component of the graph with vertex set $S_1$ and edge set $E_1(S_1)$ (where $E_1(S_1)$ consists of elements in $E_1$ that connect only elements of $S_1$). Then necessarily $M[\binom{C}{2},C] = 0$.
\end{lemma}
\begin{proof}
We first show using Fact \ref{fact:socorth} that all the vectors $\{b_j\}_{j \in C}$ are necessarily collinear. Let $j_1,j_2 \in S_1$ such that $e = \{j_1,j_2\} \in E_1$. Note that since $M_{e,j_1} = M_{e,j_2} = 0$ then we have, by Fact \ref{fact:socorth} that $b_{j_1}$ and $b_{j_2}$ are collinear. It is easy to see thus now that if $j_1$ and $j_2$ are connected by a path in the graph $(S_1,E_1(S_1))$ then $b_{j_1}$ and $b_{j_2}$ must be collinear.

We thus get that all the columns of $M$ indexed by $C$ must be proportional to each other, and so they must have the same sparsity pattern. Now let $e \in \binom{C}{2}$. If $a_e = 0$ then $M[e,C] = 0$ since the entire row $e$ is zero. Otherwise if $a_e \neq 0$ let $e=\{j_1,j_2\}$ with $j_1,j_2 \in C$. Since $M_{e,j_1} = 0$ it follows that for any $j \in C$ we must have $M_{e,j} = 0$, i.e., $M[e,C] = 0$. This is true for any $e \in \binom{C}{2}$ thus we get that $M[\binom{C}{2},C] = 0$.
\end{proof}



To complete the proof of Lemma \ref{lem:main} assume that $n \geq 3n_0^2$ for some $n_0 \geq 2$. We need to show that there is a subset $C$ of $[n]$ of size at least $n_0$ such that $M[\binom{C}{2},C] = 0$.

First note that if the graph $(S_1,E_1(S_1))$ has a connected component of size at least $n_0$ then we are done by Lemma \ref{lem:Msparsity}. Also note that if $S_0$ has size at least $n_0$ we are also done because all the columns indexed by $S_0$ are identically zero by definition.

In the rest of the proof we will thus assume that $|S_0| < n_0$ and that the connected components of $(S_1,E_1(S_1))$ all have size $ < n_0$. We will show in this case that $E_0$ necessarily contains a clique of size at least $n_0$ (i.e., a subset of the form $\binom{C}{2}$ where $|C| \geq n_0$) and this will prove our claim since all the rows in $E_0$ are identically zero by definition. The intuition is as follows: the assumption that $|S_0| < n_0$ and that the connected components of $(S_1,E_1(S_1))$ have size $< n_0$ mean that the graph $(S_1,E_1(S_1))$ is very sparse. In particular this means that $E_1$ has to be small which means that $E_0 = E_1^c$ must be large and thus it must contain a large clique.

More precisely, to show that $E_1$ is small note that it consists of those edges that are either in $E_1(S_1)$ or, otherwise, they must have at least one node in $S_1^c = S_0$. Thus we get that
\[
|E_1| \leq |E_1(S_1)| + |S_0|(n-1) \leq |E_1(S_1)| + (n_0-1) (n-1).
\]
where in the second inequality we used the fact that $|S_0| < n_0$. Also since the connected components of $(S_1,E_1(S_1))$ all have size $<n_0$ it is not difficult to show that $|E_1(S_1)| < n_0 n/2$ (indeed if we let $x_1,\dots,x_k$ be the size of each connected component we have $|E_1(S_1)| \leq \frac{1}{2} \sum_{i=1}^k x_i^2 < \frac{1}{2} \sum_{i=1}^k n_0 x_i \leq \frac{1}{2} n_0 n$).
Thus we get that
\[
|E_1| \leq \frac{n_0 n}{2} + (n_0 - 1)(n-1) \leq \left(\frac{3}{2} n_0 - 1\right) n
\]
Thus this means, since $E_0 = \binom{n}{2} \setminus E_1$:
\[
|E_0| \geq \binom{n}{2} - \left(\frac{3}{2} n_0 - 1\right) n > \frac{n^2}{2} - \frac{3}{2} n_0 n
\]
We now invoke a result of Tur{\'a}n to show that $E_0$ must contain a clique of size at least $n_0$:
\begin{theorem}[{Tur{\'a}n, see e.g., \cite{Aigner_TuranTheorem}}]
Any graph on $n$ vertices with more than $\left(1-\frac{1}{k}\right) \frac{n^2}{2}$ edges contains a clique of size $k+1$.
\end{theorem}
By taking $k = n_0-1$ we see that $E_0$ contains a clique of size $n_0$ if
\[
\frac{n^2}{2} - \frac{3}{2} n_0 n \geq \left(1-\frac{1}{n_0-1}\right) \frac{n^2}{2}
\]
This simplifies into
\[
n \geq 3n_0(n_0-1)
\]
which is true for $n \geq 3n_0^2$.
\end{proof}

\section{Slices of the $3\times 3$ positive semidefinite cone}
\label{sec:slices}

Certain slices of $\S^3_+$ are known to admit a second-order cone representation. For example the following second-order cone representation of the slice  $\{X \in \S^3_+ : X_{11} = X_{22}\}$ appears in \cite{glineurSOC_ICCOPT2013}:
\begin{equation}
\begin{bmatrix}
t & a & b\\
a & t & c\\
b & c & s
\end{bmatrix} \succeq 0
\quad
\Longleftrightarrow
\quad
\exists u, v \in \RR \text{ s.t. }
\begin{bmatrix} t+a & b+c\\ b+c & u\end{bmatrix} \succeq 0, \quad
\begin{bmatrix} t-a & b-c\\ b-c & v\end{bmatrix} \succeq 0, \quad
u+v = 2s.
\label{eq:3x3slicerep}
\end{equation}
(The $2\times 2$ positive semidefinite constraints can be converted to second-order cone constraints using~\eqref{eq:SOC2x2PSD}). To see why \eqref{eq:3x3slicerep} holds note that by applying a congruence transformation by $\frac{1}{\sqrt{2}}\left[\begin{smallmatrix} 1 & 1 & 0\\
1 & -1 & 0\\
0 & 0 & 2\end{smallmatrix}\right]$ on the $3\times 3$ matrix on the left-hand side of \eqref{eq:3x3slicerep}
we get that
\[
\begin{bmatrix}
t & a & b\\
a & t & c\\
b & c & s
\end{bmatrix}
\succeq 0
\Longleftrightarrow
\begin{bmatrix}
t+a & 0 & b+c\\
0 & t-a & b-c\\
b+c & b-c & 2s
\end{bmatrix} \succeq 0.
\]
The latter matrix has an arrow structure and thus using results on the decomposition of matrices with chordal sparsity pattern \cite{griewank1984existence,grone1984positive,agler1988positive} we get the decomposition \eqref{eq:3x3slicerep}.

On the other hand, the proof presented in the previous section can  be used to show that certain other slices do \emph{not} admit second-order cone representations. Consider the following slice which we denote by $C$:
\[
C = \{X \in \S^3_{+} : X_{11} = X_{22} + X_{33}\}.
\]
We will argue that there is no second-order cone representable set that is contained between $C$ and $\S^3_+$ (in particular $C$ does not have a second-order cone lift). To do so we first need to introduce the notion of \emph{generalized slack matrix} for a pair of convex cones $K_1 \subseteq K_2$: such matrix has its rows indexed by $\ext(K_2^*)$ (the valid linear inequalities of $K_2$) and its columns indexed by $\ext(K_1)$ and is defined by
\[
S_{K_1,K_2}[x,y] = \langle x, y \rangle \quad x \in \ext(K_2^*), y \in \ext(K_1).
\]
One can show that $S_{K_1,K_2}$ has a second-order cone factorization if and only if there exists a second-order cone representable set between $K_1$ and $K_2$ (this can be proved using exactly the same arguments as in, e.g., Theorem 4 of \cite{psdranksurvey}). 


 We now claim that matrices $A_n$ of the form given by \eqref{eq:sparsityAn} appear as submatrices of the generalized slack matrix of $C$ and $\S^3_+$. Indeed note that the vectors $v_i = (1,\cos(\theta_i),\sin(\theta_i))$ given in \eqref{eq:vistrig} satisfy $v_i v_i^T \in C$. Furthermore for any $i,j$ it is clear that $\langle (v_i \times v_j)(v_i \times v_j)^T, X \rangle \geq 0$ is a valid inequality for $\S^3_+$. It thus follows that the matrix
\[
A_n[\{i_1,i_2\},j] = \langle (v_{i_1} \times v_{i_2})(v_{i_1} \times v_{i_2})^T, v_j v_j^T \rangle = \langle (v_{i_1} \times v_{i_2}), v_j\rangle^2
\]
is a submatrix of the generalized slack matrix of $C$ and $\S^3_+$. The arguments from the previous section show that $\ranksoc(A_n)$ goes to $+\infty$ with $n$. Thus this shows that there cannot be any second-order cone representable set that lies between $C$ and $\S^3_+$.

\section*{Acknowledgments}

I would like to thank Fran\c{c}ois Glineur, Pablo Parrilo and James Saunderson for comments on a draft of this manuscript that helped improve the exposition.

\bibliographystyle{alpha}
\bibliography{../../bib/nonnegative_rank}

\end{document}